\documentclass [intlimits, 12pt]{article}

\usepackage[mathscr]{eucal}
\usepackage{amsmath}
\usepackage{amsthm}
\usepackage{amssymb}
\usepackage{amsfonts} 
\usepackage{mathrsfs}
\usepackage{amscd}
\usepackage{enumerate} 
\usepackage{color,wrapfig} 
\usepackage[pdftex]{graphicx}
\usepackage{layout}
\usepackage{float} 
\usepackage{txfonts} 
\usepackage{longtable}
\usepackage{mathtools} 

\newtheorem{Theorem}[equation]{Theorem}
\newtheorem{Proposition}[equation]{Proposition}
\newtheorem{Corollary}[equation]{Corollary}

\newtheorem{Definition}[equation]{Definition}
\newtheorem{Remark}[equation]{Remark}

\newtheorem{Problem}[equation]{Problem}

\newtheorem{Conjecture}[equation]{Conjecture}
\newtheorem{Definition and Proposition}[equation]{Definition and Proposition}
\newtheorem{Definition and Theorem}[equation]{Definition and Theorem}
\newtheorem{Definition and Remark}[equation]{Definition and Remark}
\newtheorem{Definition and Lemma}[equation]{Definition and Lemma}
\newtheorem{Definition and Example}[equation]{Definition and Example}

\newtheorem{`Theorem'}[equation]{`Theorem'}

\newtheorem*{IntroTheorem}{Theorem}

\newcommand{\draft}[1]{} 

\newcommand{\vsp}{\vspace{2mm}}

\def\epsilon{\varepsilon}
\def\phi{\varphi}

\newcommand{\al}{\alpha}

\newcommand{\de}{\delta}

\newcommand{\lam}{\lambda}

\newcommand{\si}{\sigma}

\newcommand{\deti}{{\ti{\de}}}

\newcommand{\dehat}{{\hat{\de}}}

\newcommand{\fti}{{\ti{f}}}

\def\N{{\mathbb N}}

\def\R{{\mathbb R}}

\def\Z{{\mathbb Z}}

\newcommand{\mcD}{\mathcal D}

\newcommand{\mcO}{\mathcal O}
\newcommand{\mcP}{\mathcal P}

\newcommand{\mfM}{\mathfrak M}

\newcommand{\mfS}{\mathfrak S}

\newcommand{\msC}{\mathscr C}

\newcommand{\msF}{\mathscr F}

\newcommand{\msL}{\mathscr L}

\newcommand{\msP}{\mathscr P}

\newcommand{\ra}{\rightarrow}

\newcommand{\IFF}{\Leftrightarrow}

\newcommand{\ti}{\tilde}
\newcommand{\x}{\times}

\newcommand{\beq}{\begin{equation}}
\newcommand{\eeq}{\end{equation}}
\newcommand{\beqs}{\begin{equation*}}
\newcommand{\eeqs}{\end{equation*}}

\DeclareMathOperator{\dist}{dist}

\DeclareMathOperator{\graph}{graph}

\DeclareMathOperator{\Id}{Id}

\DeclareMathOperator{\spt}{spt}

\DeclareMathOperator{\vol}{vol}

\def\slashii#1{\setbox0=\hbox{$#1$}             
\dimen0=\wd0                                 
\setbox1=\hbox{\sl/} \dimen1=\wd1            
\ifdim\dimen0>\dimen1                        
\rlap{\hbox to \dimen0{\hfil\sl/\hfil}}   
#1                                        
\else                                        
\rlap{\hbox to \dimen1{\hfil$#1$\hfil}}   
\hbox{\sl/}                               
\fi}                                         %
\def\slashiii#1{\setbox0=\hbox{$#1$}#1\hskip-\wd0\hbox to\wd0{\hss\sl/\/\hss}}
\DeclarePairedDelimiter\abs{\lvert}{\rvert}

\newcommand{\refsectionoptimaltransport}{Section \ref{sectionoptimaltransport}}

\newcommand{\refnomap}{Remark \ref{nomap}}

\newcommand{\refidonspt}{Proposition \ref{idonspt}}
\newcommand{\refccyclic}{Definition \ref{ccyclic}}
\newcommand{\refccyclicoptimal}{Theorem \ref{ccyclicoptimal}}

\newcommand{\refeqivstatements}{Proposition \ref{equivstatements}}

\newcommand{\refndimpartition}{Definition \ref{ndimpartition}}

\newcommand{\refsectionapplications}{Section \ref{sectionapplications}}

\newcommand{\refsymmetricpartition}{Definition \ref{symmetricpartition}}
\newcommand{\refreflectionoptimal}{Theorem \ref{reflectionoptimal}}
\newcommand{\refselfsymmetricnull}{Corollary \ref{selfsymmetricnull}}

\newcommand{\refsubsectioneuler}{Subsection \ref{subsectioneuler}}
\newcommand{\refallodd}{Proposition \ref{allodd}}
\newcommand{\refalldistinct}{Proposition \ref{alldistinct}}
\newcommand{\refeulercost}{Theorem \ref{eulercost}}

\begin{document}

\title{Optimal transport and integer partitions}

\author{Sonja Hohloch \\ Universiteit Antwerpen \\ {\small email: sonja.hohloch@uantwerpen.be} \\ {\small phone: +32-3-265-3231} \\ {\small fax: +32-3-265-3777}}

\date{\today}

\maketitle

\begin{abstract}
\noindent
We link the theory of optimal transportation to the theory of integer partitions.
Let $\msP(n)$ denote the set of integer partitions of $n \in \N$ and write partitions $\pi \in \msP(n)$ as $(n_1, \dots, n_{k(\pi)})$.
Using terminology from optimal transport, we characterize certain classes of partitions like symmetric partitions and those in Euler's identity
\beqs
\abs{ \{ \pi \in \msP(n) \mid \mbox{all } n_i \mbox{ distinct} \} } = \abs{ \{ \pi \in \msP(n) \mid \mbox{all } n_i \mbox{ odd} \}}.
\eeqs
Then we sketch how optimal transport might help to understand higher dimensional partitions.
\end{abstract}

\section{Introduction}

In this paper, we apply the theory of optimal transport to the classical topic of integer partitions. Both fields have been studied independently since the 18th century, but have -- up to our knowledge -- not yet been linked.


\subsection{Optimal transport and integer partitions}

The theory of optimal transport goes back to a problem posed by Monge \cite{monge} in 1781 and reformulated by Kantorovich \cite{kantorovich 42, kantorovich 48} in 1942. Monge asked the following question: Let $\mu^-$ and $\mu^+$ denote two heaps of sand with $\vol(\mu^-)= \vol(\mu^+)$ as in Figure \ref{transport}. Imagine both heaps as consisting of grains. Is there a map $\phi: \mu^- \to \mu^+$ with $\phi(\mu^-)=\mu^+$ minimizing $\sum_{x \in \mu^-} dist_{eucl}(x, \phi(x))$, i.e.\ can we `transport' $\mu^-$ into $\mu^+$ such that the sum of the transported distances is minimal? And what if we assume the heaps to consist of `continuous' matter instead of grains? Or what if we replace the Euclidean distance by a more general `cost function' $c:\mu^- \x \mu^+ \to \R$?

\begin{figure}[h]
\begin{center} 
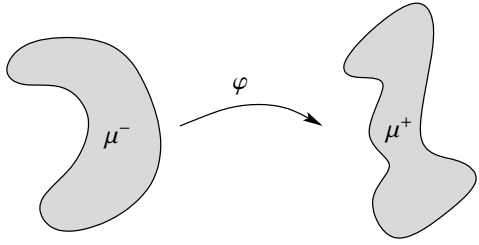
\caption{Transporting $\mu^-$ into $\mu^+$.}
\label{transport}
\end{center}
\end{figure}

Reformulated in modern language, we consider two finite measures $\mu^-$ and $\mu^+$ in $\R^n$ with $\mu^-(\R^n) = \mu^+(\R^n)$ and we are looking for a map $\phi : \R^n \to \R^n$ with image measure $\phi(\mu^-)=\mu^+$ minimizing $\int_{\R^n} c(x, \phi(x)) d\mu^-(x)$ for a given cost function $c: \R^n \x \R^n \to \R$. The infimum (hopefully minimum) is denoted by $C(\mu^-, \mu^+)$. For more details, in particular Kantorovich's formulation, and references we refer the reader to \refsectionoptimaltransport.

\vsp

Integer partitions seem to have fascinated humans already since the stone age (cf.\ Andrews $\&$ Eriksson \cite{andrews-eriksson}). Partitions describe the way to decompose an integer into (a sum of) integers: for instance, if we take the number $4$, there are the five partitions $4=3+1=2+2= 2+1+1 = 1+ 1+ 1+1$. One considers partitions up to reordering, i.e.\ $3+1$ and $1+3$ are the same partition. Denote by $\msP(n)$ the set of partitions of the (positive) integer $n \in \N$ and by $p(n)$ its cardinality. Partitions $\pi \in \msP(n)$ are usually written as ordered tuples $\pi=(n_1, \dots, n_{k(\pi)})$ where we often abbreviate $k(\pi)=k$, i.e. $\pi =(n_1, \dots, n_k) $, if no confusion is possible. For instance, we have
$\msP(4)=\{ (4), (3,1), (2,2), (2,1,1), (1,1,1,1) \}$.
There are several ways to display partitions visually, like Ferrer graphs and Young tableaux, see Figure \ref{ferrer}.

\begin{figure}[h]
\begin{center} 
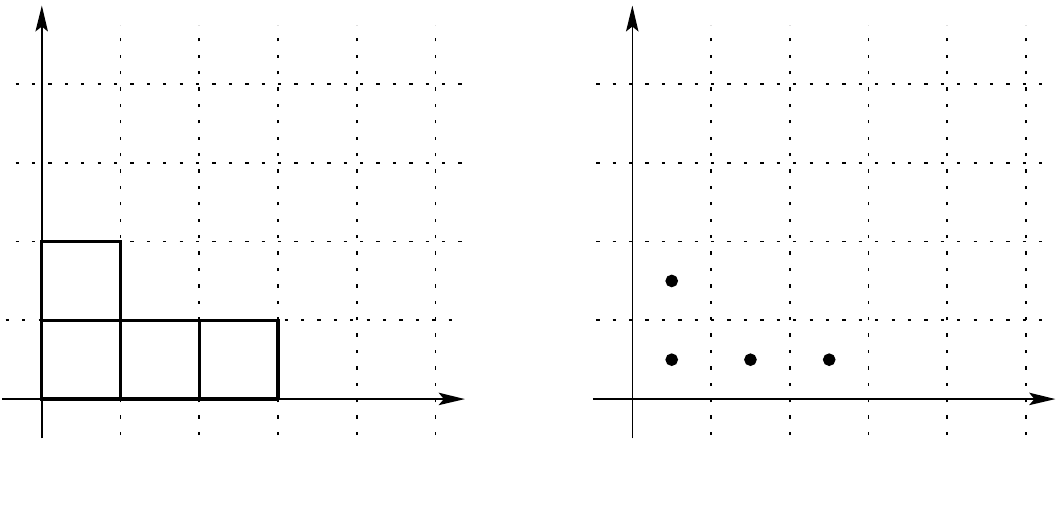
\caption{(a) Young tableau and (b) Ferrer graph of $(2,1,1) \in \msP(4)$.}
\label{ferrer}
\end{center}
\end{figure}

\vsp

Using the so-called `generating function method', Euler proved that the elements of the sequence $(p(n))_{n \in \N}$ are the coefficients of the expansion 
\beqs
\prod_{i \geq 1}\frac{1}{1-x^i}=\sum_{n \geq 0} p(n)x^n
\eeqs
which is called the {\em generating function} of $p(n)$. The generating function method is an important tool, see Andrews \cite{andrews} and Andrews $\&$ Eriksson \cite{andrews-eriksson} for more details.	
Denote by $\msP(n \mid A) \subseteq \msP(n)$ the subset of partitions with property $A$ and by $p(n \mid A)$ its cardinality. Among many other `partition identities', Euler proved by means of generating functions
\begin{equation}
\label{euleridentity}
p(n \mid \mbox{all } n_i \mbox{ odd}) = p(n \mid \mbox{all } n_i \mbox{ mutually distinct}).
\end{equation}
Integer partitions as described above are called `one dimensional' integer partitions: Intuitively one can see a Ferrer graph or a Young tableau as `graph' over the one dimensional real axis. Consequently, `two dimensional' partitions of an integer $n$ can be displayed as Ferrer graphs or Young tableaux over the two dimensional plane $\R^2$. We denote by $\msP_2(n)$ the set of two dimensional partitions of $n \in \N$ and by $p_2(n)$ its cardinality. Figure \ref{2dimyoung} (a) displays the Young tableau of
$\left[\begin{smallmatrix}
1 & \\ 2 & 1                                                                                                                                                                                                                                                                                                                                                                                                                                                                                                                                                                                                                                                                                                                                                                                                                                \end{smallmatrix} \right]
\in \msP_2(4)
$
and Figure \ref{2dimyoung} (b) shows the Young tableau of
$\left[ \begin{smallmatrix} 1 \\ 1 \end{smallmatrix} \right] \in \msP_2(2)$.

\begin{figure}[h]
\begin{center} 
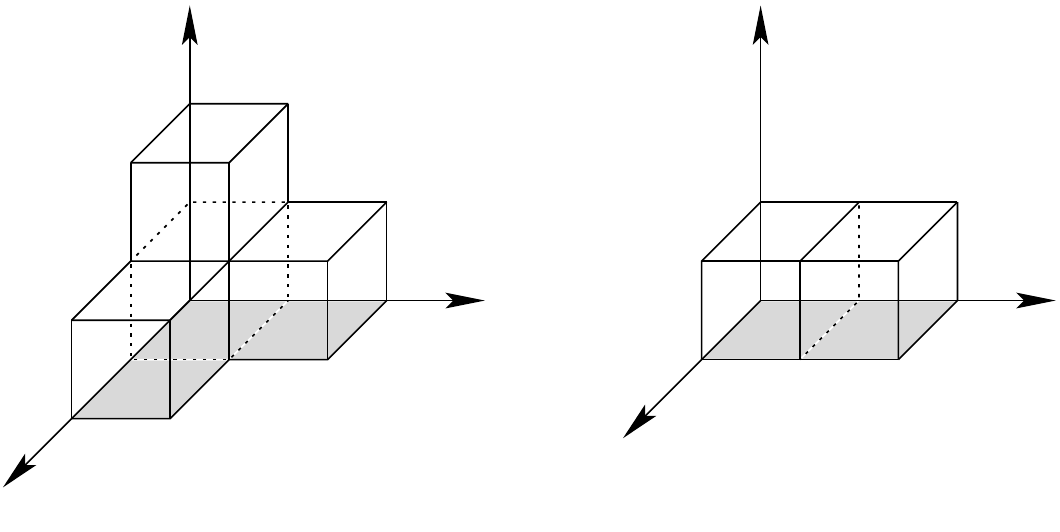
\caption{Two dimensional Young tableaux.}
\label{2dimyoung}
\end{center}
\end{figure}

Two dimensional partitions have been studied by MacMahon who found 
\beqs
\prod_{k \geq 1} \frac{1}{(1-x^k)^k} = \sum_{n\geq 0} p_2(n) x^n
\eeqs
as generating function for the two dimensional partitions (cf.\ Andrews \cite{andrews} and Andrews $\&$ Eriksson \cite{andrews-eriksson}).

\vsp

Analogously one can define the set of $m$-dimensional partitions $\msP_m(n)$ with $n \in \N$.
A look at the literature of the last 250 years shows that there has been lots of research on one dimensional partitions, some research on two dimensional partitions (also called `plane partitions'), but not much on $m$-dimensional partitions for $m \geq 3$. 
This might be due to the fact that higher dimensional partitions are difficult to display such that it is hard to obtain a good intuition -- many one dimensional partition identities have been found by manipulating the associated Ferrer graphs in a tricky way.


\subsection{Main results}

We hope to have found a way to overcome the `dimension problem' or at least make higher dimensional partitions easier to approach.
Let us first line out our ideas for one dimensional partitions. 

\vsp

We want to link optimal mass transportation to integer partitions. How do we do this? Well, let us have a look at the Ferrer board or Young tableau of a partition. Instead of thinking of a Ferrer graph as marked points in the plane, interpret it as the sum of point measures. If $\de_{(x,y)}$ denotes the point measure with mass one at $(x,y) \in \R^2$ we can see a partition $\pi=(n_1, \dots, n_k) \in \msP(n)$ as $\de_\pi:= \sum_{i=1}^k \sum_{\al=1}^{n_i} \de_{( i, \al)}$. Analogously one can proceed with the Lebesgue measure restricted to squares and Young tableaux. We have $\de_\pi(\R^2)=n$ for all $\pi \in \msP(n)$. Note that the support of each $\de_\pi$ for $\pi \in \msP(n)$ consists of exactly $n$ points. Thus, given two partitions $\pi^-$, $\pi^+ \in \msP(n)$, we can set 
\beqs
\mu^-:= \de_{\pi^-} \quad \mbox{and} \quad \mu^+:=\de_{\pi^+}
\eeqs
and look for an `optimal' map $\phi$ `transforming' $\pi^-$ into $\pi^+$ in an `optimal way' w.r.t. a given cost function $c$, i.e.\ we look for $\phi$ with $\phi(\de_{\pi^-})=\de_{\pi^+}$ and 
\beqs
C(\de_{\pi^-}, \de_{\pi^+})=  \int_{\spt(\de_{\pi^-})}c((x,y), \phi(x,y)) d \de_{\pi^-}(x,y)
\eeqs
minimal. Since the support of the involved measures is finite there is always a map realizing the minimum.

\vsp

There are several subsets of partitions which can nicely be characterized by optimal transport. For example, symmetric partitions and self-symmetric partitions (see \refsymmetricpartition\ and Figure \ref{symmetric}) are easy to describe.

\begin{IntroTheorem}
Let $\pi$ be a partition and $sym(\pi)$ its symmetric partition and the Euclidean distance the cost function. Then the function which is the identity on $\spt(\de_\pi) \cap \spt(\de_{sym(\pi)})$ and the reflection on the $x=y$ axis is optimal for $\de_\pi$ and $\de_{sym(\pi)}$. Moreover, a partition is selfsymmetric if and only if $\spt(\de_\pi) = \spt(\de_{sym(\pi)})$, i.e.\ the identity function is optimal.
\end{IntroTheorem}

This result is proven in \refreflectionoptimal\ and \refselfsymmetricnull. Now let us have a look at Euler's identity \eqref{euleridentity}.
The left hand side as well as the right hand side can be nicely described by optimal transport. 
To $\pi \in \msP_1(n)$, we associate another two types of measures. There is $\dehat_\pi$ and, for a given permutation $\si$, the measure $\de_\pi^\si$ (for details, see \refsubsectioneuler).

\begin{IntroTheorem}
Let the cost function be a metric and $S$ the reflection on the $x$-axis. Then
\begin{enumerate}[1)]
\item
$\pi \in \msP_1(n \mid all\ n_i\ odd)\quad  \Leftrightarrow \quad  S(\dehat_\pi)=\dehat_\pi \quad \Leftrightarrow \quad C(\dehat_\pi, S(\dehat_\pi))=0.$
\item
$\pi \in \msP_1(n \mid all\ n_i\ distinct)\ \Leftrightarrow\ \forall \si \neq \Id: \de_\pi \neq \de_\pi^\si \ \Leftrightarrow\ \forall \si \neq \Id: C(\de_\pi, \de_\pi^\si)\neq 0$.
\end{enumerate}
\end{IntroTheorem}

This is proven in \refallodd\ and \refalldistinct.
Andrews $\&$ Eriksson \cite{andrews-eriksson} give an explicit algorithm which turns a partition with distinct $n_i$ into a partition with only odd $n_i$, thus proving Euler's identity \eqref{euleridentity}. 
The essential part of the algorithm is expressed by a certain map $\phi$ from $\msP_1(n \mid all\ n_i\ distinct)=:\mcD$ to the slightly generalized space $\msP_1^{perm}(n \mid all\ n_i\  odd)=:\mcO$ whose exact definition is stated after \refalldistinct. In \refeulercost, we show that

\begin{IntroTheorem}
There is a cost function $\msC: \mcD \x \mcO \to \R_+$ for which $\phi$ is optimal.
\end{IntroTheorem}

We believe that an analogous statement can always be proven if there is an explicit bijection or algorithm for a partition identity.

\vsp

When identifying partitions with measures, the dimension of the partition is irrelevant: A partition $\pi \in \msP_\ell(n)$ can be seen as the measure $\de_\pi$ on $\R^{\ell+1}$ given by 
\beqs
\de_{\pi}= \sum_{i_1 =1 } ^{k_1} \dots \sum_{i_\ell=1}^{k_\ell} \sum_{\al=1}^{n_{i_1 \dots i_\ell}} \de_{(i_1, \dots, i_\ell, \al)}
\eeqs
where the $n_{i_1 \dots i_\ell}$ are monotone decreasing in each coordinate (see \refndimpartition). Similarly, we can use the Young tableaux and the restriction of the Lebesgue measure to cubes in $\R^{\ell+1}$. 
As in the one dimensional case, we can formulate the optimal transport setting for 
two partitions $\pi^-$, $\pi^+ \in \msP_\ell(n)$ via
\beqs
\mu^-:= \de_{\pi^-} \quad \mbox{and} \quad \mu^+:=\de_{\pi^+}
\eeqs
and look for a map $\phi$ with $\phi(\de_{\pi^-})=\de_{\pi^+}$ sending $\pi^-$ to $\pi^+$ in an optimal way, i.e.\ minimizing $\int_{\spt(\de_{\pi^-})}c(z, \phi(z)) d \de_{\pi^-}(z)$.

\vsp

When investigating higher dimensional partitions in future works, we would like to use results from the theory of optimal transportation. Since most results are stated for Lebesgue continuous measures $\mu^-$, it might be advisable to work with the Lebesgue measure induced from the cubes of the Young tableau rather than with the discrete measure coming form the Ferrer graph.
Moreover, the possibility to dualize the problem as done in \eqref{dual} might be of help.

\vsp

It is interesting to note, as pointed out to us by Leonid Polterovich, that discrete measures and the Monge-Kantorovich distance find applications in statistics and computer science (in particular in artificial intelligence), cf.\ Cuturi \cite{cuturi}.

\subsubsection*{Organisation of the paper}
Section 2 recalls important facts from the theory of optimal transportation. Section 3 defines integer partitions in any dimension by means of measure theory and formulates the setting for optimal transport of partitions. Section 4 applies the theory of optimal transportation to partitions.


\section{Optimal transport}

\label{sectionoptimaltransport}

The literature on mass transportation problems is vast, see for instance Villani \cite{villani} or Rachev $\&$ R\"uschendorf \cite{rachev-rueschendorf} for an overview.
In modern language, Monge's problem \cite{monge} can be formulated as follows. Let $M$ be a finite dimensional manifold and denote by $\mfS(M)$ its Borel $\sigma$-algebra and by $\mfM(M)$ the space of finite, positive Borel measures on $M$. Given a measurable map $\psi : M \to M$ and $\mu \in \mfM(M)$, the {\em image} or {\em push forward measure} $\psi(\mu)$ is defined via $\psi(\mu(B)):=\mu(\psi^{-1}(B))$ for all measurable $B \subset M$.

\begin{Problem}[Monge]
\label{monge}
\mbox{  }
\begin{longtable}{l p{100mm}}
{\bf Given:} & 
$\mu^-$, $\mu^+ \in \mfM(M)$ with $\mu^-(M)=\mu^+(M)$ and a measurable `{\textbf {\textit cost function}}' $c: M \x M \to \R^{\geq 0}$. \\
{\bf Wanted:} & 
A measurable `{\textbf {\textit optimal map}}' $\phi: M \to M$ which realizes the minimum of 
\beqs
C(\mu^-, \mu^+):= \inf\left\{ \left. \int_M c(x, \phi(x))d \mu^-(x) \right| \phi \mbox{ Borel, }\phi(\mu^-)=\mu^+\right\}.
\eeqs
\end{longtable}
\end{Problem}

In particular, in case $M=\R^m$ with point measures $\mu^-=\sum_{i=1}^\ell \de_{x_i}$ and $\mu^+=\sum_{i=1}^\ell\de_{y_i}$ and $c(x,y)=d(x,y)$ is a distance function, the map $\phi$ is transporting the mass of $\mu^-$ to the mass of $\mu^+$ while minimizing the `sum' $\sum_{i=1}^\ell d(x_i, \phi(x_i))$ of the total transported distances. But when working with point measures, we have to be cautious:

\begin{Remark}
\label{nomap}
If $\mu^-$ contains point measures, there does not necessarily exist a transport map: For $\mu^-=\de_x$ and $\mu^+=\frac{1}{2} \de_{y_1} + \frac{1}{2} \de_{y_2}$ there is no map $\phi $ with $\phi(\mu^-)=\mu^+$.
\end{Remark}

In 1979, Sudakov \cite{sudakov} proposed a proof of Monge's problem in case of $\R^m$ with the Euclidean distance as a cost function. Unfortunately the proof turned out to have a gap (cf. Ambrosio \cite[p. 137]{ambrosio1}, \cite[chapter 6]{ambrosio2}) which can only be mended under stronger assumptions.

\vspace{3mm}

Kantorovich \cite{kantorovich 42, kantorovich 48} came up with another approach to Monge's problem which is much easier to handle. Denote by $p^-: M \x M \to M$, $p^-(x,y)=x$ the projection on the first factor and by $p^+$ the projection on the second one. Let $\mu^-$, $\mu^+ \in \mfM(M)$ with $\mu^-(M)=\mu^+(M)$ and define 
\beqs
\mfM(\mu^-, \mu^+):=\{ \mu\in \mfM(M \x M) \mid p^-(\mu)=\mu^- \mbox{ and } \ p^+(\mu)=\mu^+\}.
\eeqs

\begin{Problem}[Kantorovich]
\label{kantorovich}
\mbox{  }
\begin{longtable}{l p{100mm}}
{\bf Given:} & $\mu^-$, $\mu^+ \in \mfM(M)$ and a `{\textbf {\textit cost function}}' $c: M \x M \to \R^{\geq 0}$. \\
{\bf Wanted:} & An `{\textbf {\textit optimal measure}}' $\mu \in \mfM(\mu_-, \mu_+)$ which realizes the minimum of
\beqs
K(\mu^-, \mu^+):= \inf\left\{ \left. \int_{M \x M} c(x,y) d\mu(x,y) \right| \mu \in \mfM(\mu^-, \mu^+) \right\}.
\eeqs
\end{longtable}
\end{Problem}

In contrast to Monge's problem, Kantorovich's setting is linear in $\mu$ and $\mfM(\mu^-, \mu^+)$ is convex. Moreover, under reasonable assumptions on $\mu^-$, $\mu^+$ and $c$, there exists always an optimal measure on $M$ (via a standard compactness argument using the calculus of variations).

\vspace{3mm}

Monge's and Kantorovich's problem are linked as follows. If there is a measurable map $\phi: M \to M$ with $\phi(\mu^-)=\mu^+$, then define $\Id \x \phi: M \to M \x M$, $x \mapsto (x, \phi(x))$. The image measure satisfies $(\Id \x \phi)(\mu^-) \in \mfM(\mu^-, \mu^+)$ and its support lies in the graph of $\phi$. We compute
\begin{align*}
\inf_{\phi \ with \ \phi(\mu^-)=\mu^+} \int_M c(x,\phi(x)) d\mu^-(x) 
& = \inf_{\phi \ with \ \phi(\mu^-)=\mu^+} \int_{M \x M} c(x,y) d (\Id \x \phi)(\mu^-)(x,y) \\
& \geq \min_{\mu \in \mfM(\mu^-, \mu^+)} \int_{M \x M} c(x,y) d\mu(x,y)
\end{align*}
such that Kantorovich's problem yields a lower bound for Monge's problem. The question when an optimal $\mu$ is of the form $\Id \x \phi$ was studied by Gangbo $\&$ McCann \cite{gangbo-mccann} on $\R^m$ with Lebesgue continuous $\mu^-$ for strictly convex and strictly concave cost functions. 

If the cost function is of the form $c(x,y)=h(x-y)$ where $h$ is strictly convex and satisfies certain growth conditions then Gangbo $\&$ McCann \cite{gangbo-mccann} find a unique optimal $\mu \in \mfM(\mu^-, \mu^+)$ for Kantorovich's problem which turns out to be of the form $(\Id \x \phi)(\mu^-)$. Thus they obtain also an optimal map for Monge's problem. This $\phi$ is explicitly given by $\phi(x)= x- \nabla h^{-1}(\nabla \psi(x))$ where $\psi :\R^m \ra \R \cup \{-\infty\}$ is a so-called $c$-concave function, i.e. $\psi$ is not identical $ -\infty$ and $\psi(x):= \inf_{(y, r)\in A}\{c(x,y) + r\}$ for a subset $A\subseteq \R^m \times \R$.

If $c(x,y)=f(\abs{x-y})$ (where the function $f$ is non negative) is strictly concave then the cost function induces a metric such that a minimal measure does not `move' the intersection set of the support of $\mu^-$ and $\mu^+$. Thus one only obtains a map if the two measures have disjoint support. Otherwise one only gets a map for the positive parts of the Jordan decomposition $[\mu^- -\mu^+]_+$ and $[\mu^+ -\mu^-]_+$.

\vsp

An essential tool in Gangbo $\&$ McCann's work is the so-called $c$-cyclic monotonicity (cf. \refccyclic) which was introduced by Smith $\&$ Knott \cite{smith-knott} and R\"uschendorf \cite{rueschendorf}, \cite{rachev-rueschendorf}.

\vsp

There is another way to approach Kantorovich's problem: Since it is a `convex problem' Kantorovich \cite{kantorovich 42} (and later others) dualized it, i.e. they consider 
\begin{equation}
\label{dual}
\min_{\mu \in \mfM(m^-, \mu^+)} \int_{M \x M} c(x,y) d\mu(x,y) = \sup \left\{\int_M h_-(x)d\mu^-(x) +\int_M h_+(y)d\mu^+(y)\right\}
\end{equation}
where the supremum is taken over all $(h_-, h_+) \in L^1(\mu^-) \x L^1(\mu^+)$ with $h_-(x) + h_+(y)\leq c(x,y)$.


\section{Integer Partitions}

\label{sectionintegerpartitions}

Integer partitions naturally arise in many places throughout mathematics, physics and computer sciences. They have been studied already by Euler in the 18th century and later by Legendre, Ramanujan, and Hardy to name just a few. We refer the reader to the two monographs by Andrews \cite{andrews} and Andrews $\&$ Eriksson \cite{andrews-eriksson} for further information on the history of integer partitions.

\subsection{Partitions as measures}

Since, on the one hand, 1-dimensional partitions are much more intuitive and, on the other hand, most of \refsectionapplications\ deals with 1-dimensional partitions, we first present the new concept for 1-dimensional partitions before we generalize it to the higher dimensional case.

\begin{Definition}
Let $n \in \N$. Then a {\textbf {\textit one dimensional partition}} of $n$ is a tuple of ordered integers $(n_1, \dots, n_k)$ with $n \geq n_1 \geq \dots \geq n_k \geq 1$ for some integer $1 \leq k \leq n$ such that $\sum_{i=1}^k n_i =n$. The {\textbf {\textit set of one dimensional partitions}} of $n$ is denoted by $\msP(n):=\msP_1(n)$ and $p(n):=p_1(n):=\abs{\msP_1(n)}$ is its cardinality. We often abbreviate `one dimensional partitions' by `partitions'. 
\end{Definition}

Let us introduce some notations.

\begin{Definition}
Let $n \in \N$ and $\pi\in \msP_1(n)$. We call $\abs{\pi}:=n$ the {\textbf {\textit cardinality}} of $\pi$. Instead of $\pi= (n_1, \dots, n_{k(\pi)})$ with $\sum_{i=1}^{k(\pi)} n_i =n$ we often write by abuse of notation $\pi=(n_1, \dots, n_k)$ abbreviating $k=k(\pi)$. The number $k=k(\pi)$ is the {\textbf {\textit length}} of $\pi$. 

If $E$ denotes a certain property, we define 
$$\msP(n \mid E):=\msP_1(n \mid E):=\{ \pi \in \msP_1(n) \mid \pi \ has\ property\ E\}$$ 
and denote by $p(n \mid E):=p_1(n \mid E)$ its cardinality.
\end{Definition}

For a subset $A \subset \R^\ell$, define its {\em characteristic function} $\chi_A: \R^\ell \to \{0,1\}$ via  $\chi(x)=1$ if $x \in A$ and $\chi(x)=0$ if $x \notin A$. Finite sums of characteristic functions are called {\em stair functions}. $\msL_\ell$ is the $\ell$-dimensional Lebesgue measure and $\de_z$ denotes the point measure at a point $z\in \R^\ell$. If $\mu$ is a measure then $\spt(\mu)$ denotes its support. We abbreviate $\R_+^\ell:= ([0, \infty[)^\ell$. 

\vsp

Let $p_x: \R^2 \to \R$, $(x, y) \mapsto x$ be the projection on the first and $p_y: \R^2 \to \R$, $(x, y) \mapsto y$ be the projection on the second coordinate.
The idea of this paper is to identify partitions with measures such that we can apply optimal transport theory.

\begin{Proposition}
\label{equivstatements}
The following statements are equivalent.
\begin{enumerate}[1)]
 \item  
$(n_1, \dots, n_k) \in \msP_1(n)$.
  \item
$\pi:=(n_1, \dots, n_k) \in \N^k$ and $\rho_\pi: \R_+ \to \{n_1, \dots, n_k\}$, $\rho_\pi:= \sum_{i=1}^k n_i \chi_{]i-1, i]}$ is a monotone decreasing stair function with $\int_{\R_+} \rho_\pi (x)d \msL_1(x) =n$. We set $\msL_\pi:=\rho_\pi \msL_1$.
  \item
Let $\pi:=(n_1, \dots, n_k) \in \N^k$ and denote by $\xi_{i,\al}$ the center of the square $[i-1, i] \x [\al-1, \al]$. The measure $\de_\pi:=\sum_{i = 1}^k \sum _{\al =1}^{n_i} \de_{\xi_{i \al}}$ satisfies $\de_\pi(\R^2)=n$ and the image measure $p_x(\de_\pi)$ has a monotone decreasing density function on $\frac{1}{2} + \N_0$.
  \item
Let $\pi:=(n_1, \dots, n_k) \in \N^k$. Then $\deti_\pi:=\sum_{i = 1}^k \sum _{\al =1}^{n_i} \de_{(i,\al)}$ satisfies $\deti_\pi(\R^2)=n$ and the image measure $p_x(\deti_\pi)$ has a monotone decreasing density function on $\N$.
\end{enumerate}
\end{Proposition}

The proof is obvious. Instead of using discrete measures in item 3) and 4), we can as well consider the restriction of the Lebesgue measure to the square $[i-1, i] \x [\al-1, \al]$ and obtain a continuous measure characterizing the partition.

Note that, for $\pi=(n_1, \dots, n_k) \in \msP(n)$, the push forward measure $p_x( \msL_1|_{Y(\pi)})$ has $\rho_\pi$ as density function. Conversely, $\msL_1|_{Y(\pi)}$ can be recovered from $\rho_\pi$.

\vsp

Now we consider $m$-dimensional partitions. We abbreviate multi-indices like $\left(\begin{smallmatrix} 1 \leq i_1 \leq k_1 \\ \vdots \\ 1 \leq i_m \leq k_m \end{smallmatrix} \right)$ by $1 \leq i_1, \dots, i_m \leq k_1, \dots, k_m$.

\begin{Definition}
\label{ndimpartition}
Let $n \in \N$. An {\textbf {\textit m-dimensional partition}} of $n$ is an array consisting of $n_{i_1 \dots i_m} \in \N$ where $1 \leq i_1, \dots, i_m \leq k_1, \dots, k_m$ for some natural numbers $1 \leq k_1, \dots, k_m \leq n$ such that for each index $i_j = 1, \dots, k_j$ with $1 \leq j \leq m$ the natural numbers $n_{i_1 \dots i_m}$ are a monotone decreasing sequence with $ n \geq \max_{i_j \in \{ 1, \dots, k_j\}} n_{i_1, \dots, i_m}$ and $\min_{i_j \in \{ 1, \dots, k_j\}} n_{i_1, \dots, i_m} \geq 1$ and $\sum_{i_1=1}^{k_1} \dots \sum_{i_m=1}^{k_m} n_{i_1 \dots i_m}=n$.
$\msP_m(n)$ denotes the {\textbf {\textit set of m-dimensional partitions}} and $p_m(n):=\abs{ \msP_m(n)}$ denotes its cardinality.
\end{Definition}

As an example, consider Figure \ref{2dimyoung} where (a) displays the Young tableau of
$\left[\begin{smallmatrix}
1 & \\ 2 & 1                                                                                                                                                                                                                                                                                                                                                                                                                                                                                                                                                                                                                                                                                                                                                                                                                                \end{smallmatrix} \right]
\in \msP_2(4)
$
and (b) shows the Young tableau of
$\left[ \begin{smallmatrix} 1 \\ 1 \end{smallmatrix} \right] \in \msP_2(2)$.

\vsp
A function $g: A \subseteq \R^m \to \R$ is {\em monotone decreasing} if the functions $x_i \mapsto g(x_1, \dots, x_m)$ are monotone decreasing for all $1 \leq i \leq m$. Moreover, given two monotone decreasing functions $g_1$, $g_2 : A \subseteq \R^m \to \R$, then their sum $g_1 + g_2$ is monotone decreasing.
Denote by $p_{x_1 \dots x_{m}}: \R^{m+1} \to \R^m$, $(x_1, \dots, x_{m+1}) \mapsto (x_1, \dots, x_m)$ the projection which forgets the last coordinate.

\begin{Proposition}
The following statements are equivalent.
\begin{enumerate}[1)]
 \item 
$\pi=(n_{i_1 \dots i_m})_{1 \leq i_1, \dots , i_m \leq k_1, \dots, k_m} \in \mcP_m(n)$.
  \item
Let $\pi= (n_{i_1 \dots i_m})_{1 \leq i_1, \dots , i_m \leq k_1, \dots, k_m} \in \N^{k_1 + \dots +k_m}$ and set $\rho_\pi: \R_+^m \to \{ n_{i_1 \dots i_m} \in \N \mid 1 \leq i_1, \dots , i_m \leq k_1, \dots, k_m \}$ with $\rho_\pi= \sum_{i_1 =1}^{k_1} \dots \sum_{i_m =1}^ {k_m} n_{i_1 \dots i_m} \chi_{C_{i_1 \dots i_m}}$ where $C_{i_1 \dots i_m}:= \ ]i_1 -1, i_1] \x \dots \x\ ]i_m -1, i_m]$. Then $\rho_\pi$ is a monotone decreasing stair function with $\int_{\R_+^m} \rho_\pi(x_1,\dots, x_m) d \msL_m(x_1, \dots,x_m) =n$. We set $\msL_\pi:= \rho_\pi \msL_m$.
  \item
Let $\pi= (n_{i_1 \dots i_m})_{1 \leq i_1, \dots , i_m \leq k_1, \dots, k_m} \in \N^{k_1 + \dots +k_m}$ and denote by $\xi_{i_1 \dots i_m \al}$ the center of the cube $[i_1-1, i_1] \x \dots \x [i_m-1, i_m] \x [\al-1 , \al]$. The measure $\de_\pi:= \sum_{i_1=1}^{k_1} \dots \sum_{i_m=1}^{k_m} \sum_{\al=1}^{n_{i_1 \dots i_m}} \de_{\xi_{i_1 \dots i_m\al}}$ satisfies $\de_\pi(\R^{m+1})=n$ and the image measure $p_{x_1 \dots x_m}(\de_\pi)$ has a monotone decreasing density function on $(\frac{1}{2} + \N_0)^m$.
  \item
Let $\pi= (n_{i_1 \dots i_m})_{1 \leq i_1, \dots , i_m \leq k_1, \dots, k_m} \in \N^{k_1 + \dots +k_m}$. The measure $\deti_\pi:= \sum_{i_1=1}^{k_1} \dots \sum_{i_m=1}^{k_m} \sum_{\al=1}^{n_{i_1 \dots i_m}} \de_{(i_1,\dots, i_m,\al)}$ satisfies $\deti_\pi(\R^m)=n$ and the image measure $p_{x_1 \dots x_m}(\deti_\pi)$ has a monotone decreasing density function on $ \N^m$.
\end{enumerate}
\end{Proposition}

Using \refeqivstatements\ for intuition, the proof is obvious.
Let $p_{x_1, \dots, x_{j-1}, x_{j+1}, \dots, x_m}: \R^m \to \R^{m-1}$ be the projection which forgets the $j$th coordinate.

\begin{figure}[h]
\begin{center} 
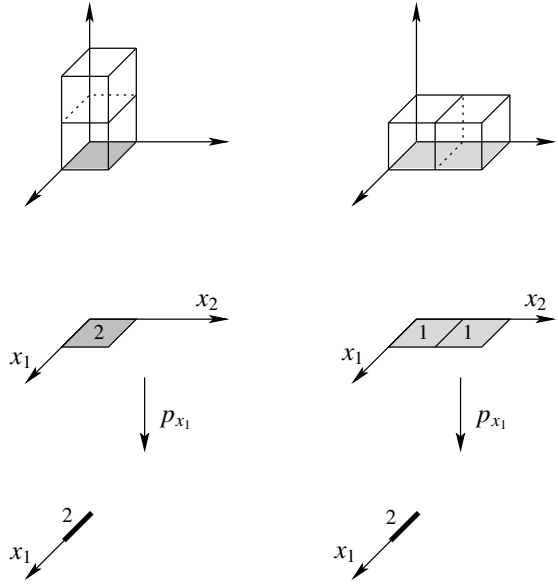
\caption{Different 2-dimensional partitions can have the same $p_{x_1}$-projection.}
\label{projections}
\end{center}
\end{figure}

\begin{Remark}
\begin{enumerate}[1)]
\item
Given a partition $\pi= (n_{i_1 \dots i_m})_{1 \leq i_1, \dots , i_m \leq k_1, \dots, k_m} \in \msP_m(n)$, the support of $\de_\pi$ is called the {\textbf {\textit Ferrer graph}} and the union of cubes $Y(\pi):=\bigcup_{i_1=1}^{k_1} \dots \bigcup_{i_m=1 }^{k_m} \bigcup_{\al=1}^{n_{i_1 \dots i_m}} [i_1-1, i_1] \x  \dots \x [i_m-1, i_m] \x [\al-1, \al]$ is called the {\textbf {\textit Young tableau}} or {\textbf {\textit Ferrer board}} of $\pi$.
 \item 
$\msL_\pi$ determines the $m$ push forward measures $p_{x_1, \dots, x_{m-1}}(\msL_\pi)$, \dots, $p_{x_2, \dots, x_m}(\msL_\pi)$ uniquely and, conversely, these $m$ push forward measures determine $\msL_\pi$ uniquely.
\item
Let $\pi \in \msP_m(n)$. Then the measures $p_{x_1, \dots, x_{m-1}}(\msL_\pi)$, \dots, $ p_{x_2, \dots, x_m}(\msL_\pi)$ represent the partitions $p_{x_1, \dots, x_{m-1}}(\pi)$, \dots, $p_{x_2, \dots, x_{m}}(\pi)$. This is of particular interest to the Kantorovich problem.
\end{enumerate}
\end{Remark}

Note that we really need all $m$ push forward measures in item 2) in order to determine $\msL_\pi$ uniquely. Otherwise it is not true, cf.\ Figure \ref{projections}.


\subsection{Optimal transport of integer partitions}

Let $m$, $n \in \N$ and $\pi \in \msP_m(n)$. The partition $\pi$ can be displayed by the continuous measure $\msL_\pi$ as well as by the discrete measures $\de_\pi$ and $\deti_\pi$. Let us first consider the latter ones. Using the above notation, we have 
\begin{align*}
\spt(\de_\pi) & =\{ \xi_{i_1 \dots i_m \al} \mid 1\leq i_1, \dots, i_m, \al \leq k_1, \dots, k_m, n_{i_1 \dots i_m} \}, \\
\spt(\deti_\pi)& =  \{ (i_1, \dots, i_m, \al) \mid 1\leq i_1, \dots, i_m, \al \leq k_1, \dots, k_m, n_{i_1 \dots i_m} \}
\end{align*}
which are sets of isolated points. For all $\pi \in \msP_m(n)$, the cardinality of the supports is $\abs{\spt(\de_\pi)}=n=\abs{\spt(\deti_\pi)}$. Thus, given two partitions $\pi^-$, $\pi^+ \in \msP_m(n)$, there {\em always exists} a bijective map $f: \spt(\pi^-) \to \spt(\pi^+)$, i.e. there are no  obstructions as in \refnomap. We denote the space of such maps by $\msF(\pi^-, \pi^+)$ and its cardinality is {\em finite}.
In this context, Monge's problem translates into looking for 
\begin{align*}
C(\pi^-, \pi^+):= \min_{f \in \msF(\pi^-, \pi^+)} \int_{\spt(\de_{\pi^-})} c(z, f(z)) d \de_\pi(z) 
= \min_{f \in \msF(\pi^-, \pi^+)} \sum_{z \in \spt(\de_{\pi^-})} c(z, f(z))
\end{align*}

We will be working a lot with the following type of cost functions.

\begin{Definition}
A measurable cost function $c: \R^m \x \R^m \to \R$ is \textbf{\textit{metric-like}} if $c$ induces a metric on $\R^m$, i.e. 
\begin{enumerate}[a)]
 \item 
$c \geq 0$ with $c(x,y)=0$ if and only if $x=y$.
  \item
$c(x,y)=c(y,x)$ for all $x$, $y \in \R^m$.
  \item
$c(x,z) \leq c(x,y) + c(y,z)$ for all $x$, $y$, $z \in \R^m$.
\end{enumerate}
\end{Definition}

Gangbo $\&$ McCann \cite{gangbo-mccann}, using Lebesgue continuous measures and concave cost functions, showed that the optimal map can be chosen to be the identity on the intersection set $\spt(\mu^-) \cap \spt(\mu^+)$. This holds true for our metric-like cost functions: 

\begin{Proposition}
\label{idonspt}
Let $\pi^-$, $\pi^+ \in \msP_m(n)$ and let $c$ be a metric-like cost function. Then an optimal $f \in \msF(\pi^-, \pi^+)$ can be chosen to be the identity on $\spt(\de_{\pi^-}) \cap \spt(\de_{\pi^+})$.
\end{Proposition}

\begin{proof}
We study where the points in $\spt(\de_{\pi^-})$ are mapped to and if one can minimize the transport costs.

{\em Case 1:} Let $z \in \spt(\de_{\pi^-})$ be mapped to $ f(z) \in \spt(\de_{\pi^-}) \cap \spt(\de_{\pi^+})$. Since $f(z) \in \spt(\de_{\pi^-})$, it is mapped to $f(f(z))$. If $f(f(z)) \neq f(z)$, then we have $c(z, f(z)) + c(f(z), f(f(z))) \geq c(z, f(f(z))) + c(f(z), f(z))= c(z, f(f(z)))$. Now define a new map $\fti \in \msF(\pi^-, \pi^+)$ via $\fti(z):= f(f(z))$ and $\fti(f(z))=f(z)$. We obtain $\int_{\spt(\de_{\pi^-})} c(z,f(z)) d \de_{\pi^-}(z) \geq \int_{\spt(\de_{\pi^-})} c(z,\fti(z)) d \de_{\pi^-}(z)$. Iterating this procedure, we obtain a function which leaves all points in $\spt(\de_{\pi^-}) \cap \spt(\de_{\pi^+})$ fixed and has equal or lower transport costs than any function which does not fix the points in $\spt(\de_{\pi^-}) \cap \spt(\de_{\pi^+})$.

{\em Case 2:} Let $z \in \spt(\de_{\pi^-})$ be mapped to $ f(z) \in \spt(\de_{\pi^+}) \setminus \spt(\de_{\pi^-})$. The point $f(z)$ is not in the support of $\de_{\pi^-}$, so $f(z)$ is mapped nowhere by $f$. There is no need to modify $f$.
\end{proof}

Moreover, we can define a distance for partitions.

\begin{Remark}
Let $\pi^-$, $\pi^+ \in \msP_m(n)$ and let $c$ be a metric-like cost function. Then $\dist(\pi^-, \pi^+):= \min_{f \in \msF(\pi^-, \pi^+)} \int_{\spt(\de_{\pi^-})} c(z, f(z)) d \de_\pi(z) $ is a metric on $\msP_m(n)$.
\end{Remark}

The following definition is important. It goes back to Smith $\&$ Knott \cite{smith-knott} and R\"uschendorf \cite{rueschendorf}.

\begin{Definition}
\label{ccyclic}
Let $c$ be a metric-like cost function. A subset $B \subset \R^m \x \R^m$ is {\textbf{\textit{$c$-cyclic monotone}}} if $\sum_{i=1}^k c(x_i, y_i) \leq \sum_{i=1}^k c(x_{\si(i)}, y_i)$ for all $k \in \N$ and $(x_i, y_i) \in B$ for $1 \leq i \leq k$ and $\si \in Perm(k)$.
\end{Definition}

Gangbo $\&$ McCann \cite{gangbo-mccann} proved that the support of an optimal measure for the Kantorovich problem is $c$-cyclic monotone and that the graph of an optimal map of the Monge problem always lies in a $c$-cyclic monotone set. The following is the according statement for our situation.

\begin{Theorem}
\label{ccyclicoptimal}
Let $\pi^-$, $\pi^+ \in \msP_m(n)$ and let $c$ be a metric-like cost function. Then $\phi \in \msF(\pi^-, \pi^+)$ is optimal for $C(\pi^-, \pi^+)$ if and only if $\graph (\phi)$ is $c$-cyclic monotone.
\end{Theorem}


\section{Optimal transport applied to integer partitions}

\label{sectionapplications}

\subsection{Symmetric partitions in $\msP_1(n)$}

Let us consider a special type of partitions.

\begin{Definition}
\label{symmetricpartition}
Let $T: \R^2 \to \R^2$, $(x, y) \mapsto (y,x)$ be the reflection on the $x=y$ line and let $\pi \in \msP_1(n)$ with Young tableau $Y(\pi)$. The {\textbf {\textit symmetric}} or {\textbf {\textit conjugate partition}} $sym(\pi) \in \msP_1(n)$ of $\pi$ is the partition with the Young tableau $Y(sym(\pi))=T(Y(\pi))$, i.e. we obtain $sym(\pi)$ by reflecting $\pi$ on the $x=y$ line. Partitions $\pi \in \msP_1(n)$ with $sym(\pi)=\pi$ are called {\textbf {\textit self-symmetric}} or {\textbf {\textit self-conjugate}}.
\end{Definition}

An example for a partition and its symmetric partition is sketched in Figure \ref{symmetric}.

\vsp

\begin{figure}[h]
\begin{center} 
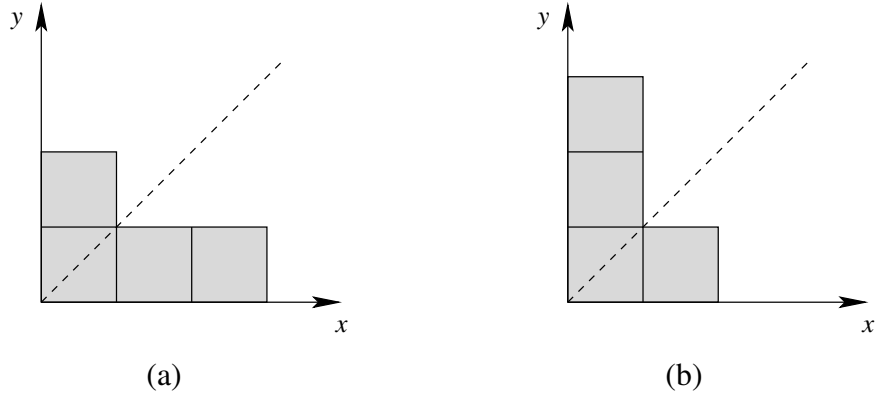
\caption{A partition and its symmetric partition.}
\label{symmetric}
\end{center}
\end{figure}

Let $p_x: \R^2 \to \R$, $(x,y) \mapsto x$ be the projection on the $x$-axis and $p_y: \R^2 \to \R$, $(x,y) \mapsto y$ the projection on the $y$-axis. 

\begin{Remark}
\begin{enumerate}[(1)]
\item
Consider $\pi =(n_1, \dots, n_k) \in \msP_1(n)$ with associated measure $\de_\pi=\sum_{i=1}^k \sum_{\al=1}^{n_i} \de_{\xi_{i \al}}$. Then $sym(\pi)$ has the associated measure $\de_{sym(\pi)}=\sum_{i=1}^k \sum_{\al=1}^{n_i} \de_{T(\xi_{i \al })}=\sum_{i=1}^k \sum_{\al=1}^{n_i} \de_{\xi_{\al i}}$.
\item
The push forward measures $p_y(\msL_\pi)$ resp. $p_y(\de_\pi)$ have the same density functions as $p_x(\msL_{sym(\pi)})$ resp. $p_x(\de_{sym(\pi)})$.
\end{enumerate}
\end{Remark}

Now we investigate the relation between a partition and its symmetric partition by means of optimal transport. We are looking for a function which realizes minimal transport costs in Monge's problem for the partitions $\pi$ and $sym(\pi)$, more precisely for the measures $\mu^-=\de_\pi$ and $\mu^+=\de_{sym(\pi)}$.

\begin{Theorem}
\label{reflectionoptimal}
Let $\pi \in \msP_1(n)$ and let $c$ be the Euclidean distance. Then $f \in \msF(\de_\pi, \de_{sym(\pi)})$ given by $f= \Id$ on $\spt(\de_\pi) \cap \spt(\de_{sym(\pi)})$ and $f=T$ elsewhere minimizes $C(\de_\pi, \de_{sym(\pi)})$. 
\end{Theorem}

\begin{figure}[h]
\begin{center} 
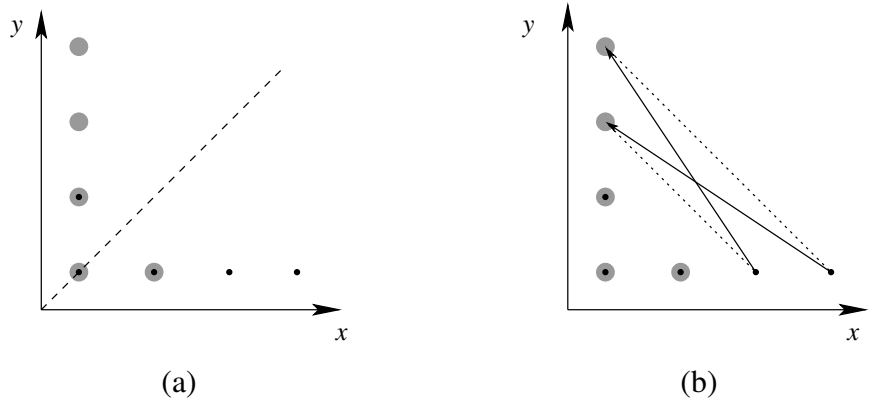
\caption{(a) $\spt(\de_\pi)$ is displayed by small black points and $\spt(\de_{sym(\pi)})$ by larger greyish points. (b) $f=T$ on $\spt(\de_\pi) \cap \spt(\de_{sym(\pi)})$ due to $c$-cyclicity.}
\label{symmetricproof}
\end{center}
\end{figure}

\begin{proof}
Due to \refidonspt, we can set $f=\Id$ on $\spt(\de_\pi) \cap \spt(\de_{sym(\pi)})$. Now consider $z \in \spt(\pi) \setminus \spt(sym(\pi))$. 
If we consider the graph of $T$, we observe that it is $c$-cyclic monotone since none of the segments between $z$ and $T(z)$ cross each other, cf.\ Figure \ref{symmetricproof}. Thus \refccyclicoptimal\ implies the claim.
\end{proof}

\begin{Corollary}
\label{selfsymmetricnull}
$\pi \in \msP_1(n)$ is self-symmetric if and only if $C(\de_\pi, \de_{sym(\pi)})=0$ with the Euclidean distance as cost function.
\end{Corollary}


\subsection{Generalized symmetric partitions in $\msP_k(n)$}

Symmetric partitions have a natural generalization to higher dimensions. For $m \in \N$, denote by $Perm(m)$ the permutation group of the set $\{1, \dots, m\}$. Given $\si \in Perm(m)$, we write $\si(1, \dots, m)=(\si(1), \dots, \si(m))$.

\begin{Definition}
\label{sigmasymmetric}
Let $\si \in Perm(k+1)$ and $T_\si: \R^{k+1} \to \R^{k+1}$ be linear with $T_\si=(e_{\si(1)}, \dots, e_{\si(k+1)})$ as matrix w.r.t. the standard basis $e_1$, \dots, $e_{k+1} $ of $\R^{k+1}$.
The $\si${\textbf {\textit -symmetric partition}} of $\pi \in \msP_k(n)$ is the partition $sym_\si(\pi)$ with the Young tableau $Y(sym_\si(\pi))=T_\si(Y(\pi))$. Partitions $\pi \in \msP_k(n)$ with $\pi=sym_\si(\pi)$ are called $\si${\textbf {\textit -selfsymmetric}}.
\end{Definition}


We believe that $\si$-symmetric partitions behave similar as symmetric partitions:

\begin{Conjecture}
 \begin{enumerate}[1)]
\item
The map $T_\si$ induces an optimal transport map for $\pi$ and $sym_\si(\pi)$.
\item
$\pi \in \msP_k(n)$ is $\si$-selfsymmetric if and only if $C(\de_\pi, \de_{sym_\si(\pi)})=0$ where $c$ is a metric-like cost functions.
\end{enumerate}
\end{Conjecture}


\subsection{The Euler identity}
\label{subsectioneuler}

When studying subsets of $\msP_1(n)$, Euler proved the identity
\begin{equation*}
p_1(n \mid all\ n_i\ odd) 
= p_1(n \mid all\ n_i\ mutually\ distinct).
\end{equation*}

We investigate if these subsets can be characterized by means of optimal transport. For $r \in \R$, set $\lfloor r \rfloor:= \max \{l \in \Z \mid l \leq r\}$.
Let $\pi=(n_1, \dots, n_k) \in \msP_1(n)$ and introduce
\beqs
\dehat_\pi:= \sum_{i=1}^k \sum_{\al=1}^{n_i} \de_{(i, \lfloor -\frac{n_i}{2} \rfloor +\al)}
\eeqs
which displays $\pi$ centered on the $x$-axis, see Figure \ref{centermeasure}.
\begin{figure}[h]
\begin{center} 
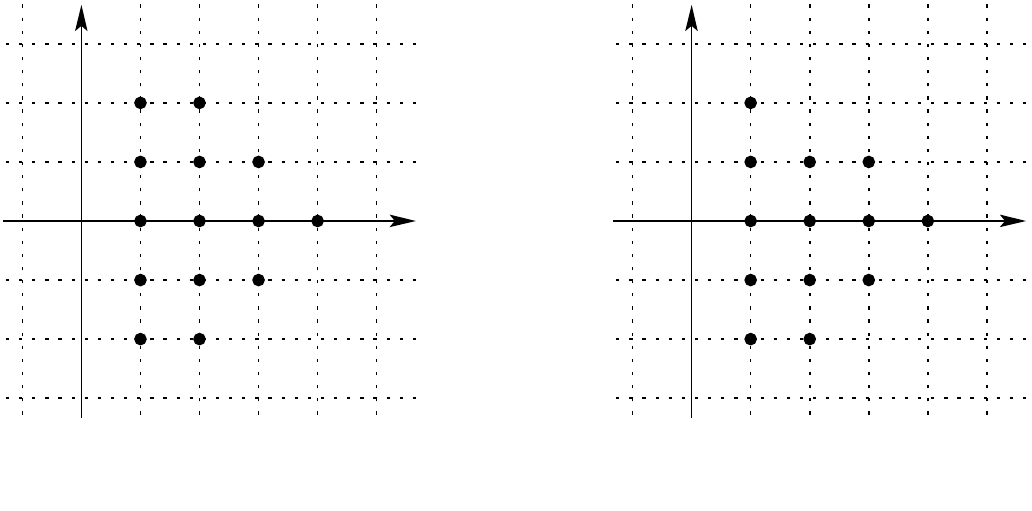
\caption{(a) The support of $\dehat_\pi$ for $\pi =(5,5, 3,1) $. (b) The support of $\dehat_{\pi}$ for $\pi=(5,4,3,1)$.}
\label{centermeasure}
\end{center}
\end{figure}
We obtain

\begin{Proposition}
\label{allodd}
Let $T: \R^2 \to \R^2$, $(x,y) \mapsto (x, -y)$ be the reflection on the $x$-axis and let $c$ be a metric-like cost function. Then
\begin{enumerate}[(1)]
 \item 
$\pi \in \msP_1(n \mid all\ n_i\ odd)$ if and only if $T(\dehat_\pi)=\dehat_\pi$.
\item
$\pi \in \msP_1(n \mid all\ n_i\ odd)$ if and only if $C(\dehat_\pi, T(\dehat_\pi))=0$.
\end{enumerate}
\end{Proposition}

\begin{proof}
(1) follows from the definition of $\dehat_\pi$. (2) follows from \refidonspt\ since $\spt(\dehat_\pi)=\spt(T(\dehat_\pi))$ if and only if $\pi \in \msP_1(n \mid all\ n_i\ odd)$.
\end{proof}

Let $\pi=(n_1, \dots, n_k) \in \msP_1(n)$ and $\si \in Perm(k)$. We define 
\begin{equation*}
\de_\pi^\si:=\sum_{i=1}^k \sum_{\al=1}^{n_i} \de_{(\si(i), \al)}.
\end{equation*}
For $\si=\Id$, we recover $\de_\pi^{\Id} = \de_\pi$.

\begin{Proposition}
\label{alldistinct}
Let $c$ be a metric-like cost function. Then
\begin{enumerate}[(1)]
 \item 
$\pi \in \msP_1(n \mid not\ all\ n_i\ mutually\ distinct)$. 

$\IFF$ There is $\si \in Perm(k(\pi))\setminus \{\Id\}$ with $\de_\pi= \de_\pi^\si$. 

$\IFF$ There is $\si \in Perm(k(\pi))\setminus \{\Id\}$ with $C(\de_\pi, \de_\pi^\si)=0$.

\item

$\pi \in \msP_1(n \mid all\ n_i\ mutually\ distinct)$.

$\IFF$ For all $\si \in Perm(k(\pi))\setminus \{\Id\}$ holds $\de_\pi \neq \de_\pi^\si$.

$\IFF$ For all $\si \in Perm(k(\pi))\setminus \{\Id\}$ holds $C(\de_\pi, \de_\pi^\si)\neq 0$.
\end{enumerate}
\end{Proposition}

\begin{proof}
(1) $\pi \in \msP_1(n \mid not\ all\ n_i\ mutually\ distinct)$ if and only if there is $1 \leq i_1 \leq i_2 \leq k(\pi)$ with $n_{i_1}=n_{i_2}$ if and only if there is $\si \in Perm(k(\pi))$, namely the permutation exchanging $i_1$ and $i_2$ and leaving the other indices fixed, with $\de_\pi= \de_\pi^\si$ if and only if $C(\de_\pi, \de_\pi^\si)=0$ due to \refidonspt.

(2) is the negation of the first item.
\end{proof}

Now we want to see how we can characterize Euler's identity \eqref{euleridentity} by means of optimal transport. Andrews $\&$ Eriksson \cite{andrews-eriksson} give an algorithm in order to show how a partition in $\msP_1(n \mid all\ n_i\ mutually\ distinct)$ is transformed into a partition in $\msP_1(n \mid all\ n_i\ odd)$ and back. We will write this algorithm as an explicit bijection, but first we need some notation. Recall that each natural number $m \in \N$ has a unique decomposition into prime factors 
\beqs
m=\prod_{p \ prime} p^{\lam(m,p)}= 2^{\lam(m,2)} \prod_{2\ <\ p \ prime} p^{\lam(m,p)}.
\eeqs
This decomposition into even and odd part suggests the definition of the functions
\begin{align*}
& g: \N \to \N, \qquad g(m):= 2^{\lam(m,2)}, \\ 
& u: \N \to \N, \qquad u(m):= \prod_{2 \ < \ p \ prime} p^{\lam(m,p)}.
\end{align*}
Note that, for $\pi=(n_1, \dots, n_{k(\pi)}) \in \msP(n)$, the number $k(\pi)$ is the number of `columns' $n_i$ of $(n_1, \dots, n_{k(\pi)})$. We introduce
\beqs
\msP_1^{perm}:=\{ \si(\pi):= (n_{\si(1)}, \dots, n_{\si(k)}) \mid \pi=(n_1, \dots, n_k) \in \msP_1(n),\ \si \in Perm(k)\}
\eeqs
which are `generalized partitions' since the $n_i$ are {\em not necessarily} monotone decreasing. 
We set
\beqs
\phi \ : \  \msP_1(n \mid all\ n_i\ mutually\ distinct) \quad  \longrightarrow \quad \msP_1^{perm}(n \mid all\ n_i\ odd)
\eeqs
given by
\begin{align*}
\begin{array}{rccc}
\phi(\pi)=\phi(n_1, \dots, n_k) := & \Bigl(\underbrace{u(n_1), \dots, u(n_1)}, & \dots & ,\underbrace{u(n_k), \dots, u(n_k)} \Bigr). \\
& g(n_1) && g(n_k)
\end{array}
\end{align*}
$\phi(\pi)$ is only a `generalized partition' since its entries $u(n_1), \dots, u(n_k)$ are not necessarily monotonically decreasing. Ordering them monotonically decreasing leads to a partition which we call $\bar{\phi}(\pi)$. It lies by construction in $\msP_1(n \mid all\ n_i\ odd)$. Andrews $\&$ Eriksson \cite{andrews-eriksson} show that this construction induces a bijection
\beqs
\bar{\phi} \ : \ \msP_1(n \mid all\ n_i\ mutually\ distinct) \quad \longrightarrow \quad \msP_1(n \mid all\ n_i\ odd).
\eeqs
They also give the inverse construction, but we do not need it here.

\begin{Theorem}
\label{eulercost}
Abbreviate $\msP_1(n \mid all\ n_i\ mutually\ distinct)=: \mcD$ and $\msP_1^{perm}(n \mid all\ n_i\ odd)=: \mcO$ and denote by $\msF(\mcD, \mcO)$ the space of maps from $\mcD$ to $\mcO$. Let $c$ be a metric-like cost function.
Then there is a cost function $\msC: \mcD \x \mcO \to \R_+$ such that $\msC(\pi, \phi(\pi)) = \min\{ C(\pi, \psi(\pi)) \mid \psi \in \msF(\mcD,\mcO) \}$. 
\end{Theorem}

\begin{proof}
The idea is to construct $\msC$ in such a way that $\msC(\pi, \phi(\pi))=0$ and $\msC(\pi, \psi(\pi))>0$ for $\psi \neq \phi$. We note $k(\phi(\pi)) \geq k(\pi)$. 
Given $\pi =(n_1, \dots, n_{k(\pi)}) \in \mcD$ we write (abusing the notation)
\beqs
\begin{array}{rc}
\phi(n_i) := & (\underbrace{u(n_i), \dots, u(n_i)})\\
& g(n_i)
\end{array}
\eeqs
and we have $k(\phi(n_i))=1$ if and only if $g(n_i)=2^0=1$.
Moreover, for $i \in \N$, define the entry functions $i: \msP_1(n) \to \N$, $i(\pi)=i(n_1, \dots, n_{k(\pi)}):=n_i$ if $1 \leq i \leq k(\pi)$ and $0$ otherwise. 
For $\pi \in \msP_1(n)$ define the `column measure'
\beqs
\de_{i(\pi)}:=\sum_{\al=1}^{i(\pi)}\de_{(i,\al)}
\eeqs
and the `odd-even prime decomposition measure' of an entry $i(\pi)$
\beqs
\de_{gu(i(\pi))}:= \sum_{j_1=1 }^{g(i(\pi))} \sum_{j_2=1}^{u(i(\pi))} \de_{(j_1, j_2)}.
\eeqs
Now pick $\psi \in \msF(\mcD,\mcO)$ and apply it to $\pi \in \mcD$. $\psi(\pi)\in \mcO$ is a generalized partition, but we can assign a measure $\de_{\psi(\pi)}$ analogously to partitions. Then $\psi(\de_\pi)=\de_{\psi(\pi)}$. But $\psi(\de_{i(\pi)})$ is not even a `generalized partition' since its support lies just somewhere in the support of $\de_{\psi(\pi)}$. Worse, we do not even know which point goes where since the push forward of a sum of point measures does not care about it as long as the supports are mapped bijectively to each other. To circumvent this problem, we will minimize over all possibilities in the following way. Let $c$ be a metric-like cost function. Set 
\beqs
d(\phi(n_i), \psi(n_i)):= \min C(\de_{\phi(n_i)}, \de_{\psi(n_i)})
\eeqs
where the minimum is taken over all (finite) possibilities to place $\psi(n_i)$ in $\psi(\pi)$. Then $d(\phi(n_i), \psi(n_i))=0$ if and only if the support of the measures $\de_{\phi(n_i)}$ and $ \de_{\psi(n_i)}$ coincides. Now define $d(\phi(\pi), \psi(\pi)):= \sum_{1 \leq i \leq \ell(\pi)} d(\phi(n_i), \psi(n_i))$. Then $d(\phi(\pi), \phi(\pi)))=0$ and $d(\phi(\pi), \psi(\pi)))\geq 0$ for $\psi \neq \phi$. Now set $\msC(\pi, \psi(\pi)):=d(\phi(\pi), \psi(\pi))$ and the claim follows.
\end{proof}


\end{document}